\documentclass[12pt]{article}
\usepackage{amsmath,amsxtra,latexsym,amsthm,amssymb,amscd,amsfonts}
\usepackage[portrait, top=3.5cm, bottom=3cm, left=3.5cm, right=2cm] {geometry}

\theoremstyle{plain}
\newtheorem{lemma}{\bf Lemma}[section]
\newtheorem{ex}{\bf Example}[section]

\newtheorem{theorem}{\bf Theorem}[section]

\def\tto{\;{\lower 1pt \hbox{$\rightarrow$}}\kern -10pt
\hbox{\raise 2pt \hbox{$\rightarrow$}}\;}

\def\epi{\mbox{\rm epi}\,}

\def\dom{\mbox{\rm dom}\,}
\def\ker{\mbox{\rm ker}\,}

\def \"{a}

\begin{document}
\pagestyle{myheadings}

\newtheorem{Theorem}{Theorem}[section]
\newtheorem{Proposition}[Theorem]{Proposition}
\newtheorem{Remark}[Theorem]{Remark}
\newtheorem{Lemma}[Theorem]{Lemma}
\newtheorem{Corollary}[Theorem]{Corollary}
\newtheorem{Definition}[Theorem]{Definition}
\newtheorem{Example}[Theorem]{Example}
\renewcommand{\theequation}{\thesection.\arabic{equation}}
\normalsize

\setcounter{equation}{0}

\title{{Differential stability of a class of convex optimal control problems}}

\author{D. T. V.~An\footnote{Department of Mathematics and Informatics, College of Sciences,
Thai Nguyen University, Thai Nguyen City, Vietnam; email: duongvietan1989@gmail.com.}\ \, and \ J.-C. Yao\footnote{ Center for General Education, China Medical University, 
			Taichung 40402, Taiwan; email: yaojc@mail.cmu.edu.tw.} \ \, and \ \, N. D. Yen\footnote{Institute
of Mathematics, Vietnam Academy of Science and Technology, 18 Hoang
Quoc Viet, Hanoi 10307, Vietnam; email: ndyen@math.ac.vn.}}

\maketitle
\date{}
\medskip
\begin{quote}
\noindent {\bf Abstract.} A parametric constrained convex optimal control problem, where the initial state is perturbed and the linear state equation contains a noise, is considered in this paper. Formulas for computing the subdifferential and the singular subdifferential of the optimal value function at a given parameter are obtained by means of some recent results on differential stability in mathematical programming. The computation procedures and illustrative examples are presented. 


\medskip


\end{quote}

\section{Introduction}
\markboth{\centerline{\it Introduction}}{\centerline{\it D.T.V.~An, J.-C.~Yao,
and N.D.~Yen}} \setcounter{equation}{0}
According to Bryson \cite[p.~27, p.~32]{Bryson}, optimal control had its origins in the calculus of variations in the 17th century. The calculus of variations was developed further in the 18th by
L.~Euler and J.~L.~Lagrange and in the 19th century by A.~M.~Legendre, C.~G.~J.~Jacobi, W.~R.~Hamilton, and K.~T.~W.~Weierstrass. In 1957, R.~E.~Bellman gave a new view of Hamilton-Jacobi theory which he called \textit{dynamic programming}, essentially a nonlinear feedback control scheme. McShane \cite{McShane} and Pontryagin et al. \cite{Pontryagin_etal.__1962} extended the calculus of variations to handle control variable inequality constraints. \textit{The Maximum Principle} was enunciated by Pontryagin. 

\medskip
As noted by Tu \cite[ p.~110]{Tu}, although much pioneering work had been carried out by other authors, Pontryagin and his associates are the first ones to develop and present the Maximum Principle in unified manner. Their work attracted great attention among mathematicians, engineers, economists, and spurred wide research activities in the area (see \cite[Chapter~6]{Mordukhovich_2006b}, \cite{Tu,Vinter2000}, and the references therein).

\medskip
Differential stability of parametric optimization problems is an important topic in variational analysis and optimization. In 2009, Mordukhovich et al.  \cite{MordukhovichNamYen2009} presented formulas for computing and estimating the Fr\'echet subdifferential, the Mordukhovich subdifferential, and the singular subdifferential of the optimal value function. If the problem in question is convex, An and Yen \cite{AnYen} and then An and Yao \cite{AnYao} gave formulas for computing  subdifferentials of the optimal value function by using two versions of the Moreau-Rockafellar theorem and appropriate regularity conditions.
It is worthy to emphasize that differential properties of optimal value functions in parametric mathematical programming and the study on differential stability of optimal value functions in optimal control have attracted attention of many researchers; see, e.g., \cite{CerneaFrankowska2005,Chieu_Kien_Toan,MoussaouiSeeger1994,TR2004,RockafellarWolenski2000I,RockafellarWolenski2000II,Seeger1996,Toan_2015,Toan_Kien} and the references therein.

\medskip
Very recently, Thuy and Toan \cite{Toan_Thuy} have obtained a formula for computing the subdifferential of the optimal value function to a parametric unconstrained convex optimal control problem with a convex objective function and linear state equations. 

\medskip
The aim of this paper is to develop the approach of \cite{Toan_Thuy} to deal with \textit{constrained control problems}. Namely, based on the paper of An and Toan \cite{AnToan} about differential stability of parametric convex mathematical programming problems, we will get new results on computing the subdifferential and the singular subdifferential of the optimal value function. Among other things, our result on computing the subdifferential extends and improves the main result of \cite{Toan_Thuy}. Moreover, we also describe in details the process of finding vectors belonging to the subdifferential (resp., the singular subdifferential) of the optimal value function. Thus, on one hand, our results have the origin in the study of \cite{Toan_Thuy}. On the other hand, they are the results of deepening that study for the case of constrained control problems. Meaningful examples, which have the origin in \cite[Example~1, p.~23]{Pontryagin_etal.__1962}, are designed to illustrate our results. In fact, these examples constitute an indispensable part of the present paper.

\medskip
Note that differentiability properties of the optimal value function in both unconstrained and constrained control problems have been studied from different point of views. For instance, Rockafellar \cite{TR2004} investigated the optimal value of a parametric optimal control problem with an differential inclusion and a point-wise control constraint as a function of the time horizon and the terminal state. Meanwhile, based on an epsilon-maximum principle of Pontryagin type, Moussaoui and Seeger \cite[Theorem~3.2]{MoussaouiSeeger1994} considered an optimal control problem with linear state equations and gave a formula for the subdifferential of the optimal value function without assuming the existence of optimal solutions to the unperturbed problem.

\medskip
The organization of the paper is as follows. Section 2 formulates the parametric convex optimal control problem which we are interested in.  The same section reviews some of the standard facts on functional analysis \cite{KF70,Rudin_1991}, convex analysis \cite{Ioffe_Tihomirov_1979}, variational analysis \cite{Mordukhovich_2006a}, and presents one theorem from \cite{AnToan} which is important for our investigations. Formulas for the subdifferential and the singular subdifferential of the optimal value function of the convex optimal control problem are established in Section 3. The final section presents a series of three closely-related illustrative examples.

\section{Preliminaries}
\markboth{\centerline{\it Preliminaries}}{\centerline{\it D.T.V.~An, J.-C.~Yao,
and N.D.~Yen}} \setcounter{equation}{0}

Let $W^{1,p} ([0,1],\Bbb{R}^n)$, $1 \le p < \infty$, be the Sobolev space  consisting of absolutely continuous functions $x:[0,1] \rightarrow \Bbb{R}^n$ such that $\dot{x} \in L^p([0,1], \Bbb{R}^n)$. Let there be given
\par - matrix-valued functions $A(t)=(a_{ij}(t))_{n\times n},\, B(t)=(b_{ij}(t))_{n\times m},$ and $ C(t)=(c_{ij}(t))_{n\times k};$
\par - real-valued functions $g: \Bbb{R}^n \to \Bbb{R}$ and $L:[0,1]\times \Bbb{R}^n \times \Bbb{R}^m \times \Bbb{R}^k \to \Bbb{R}$;
\par - a convex set $\mathcal{U} \subset L^p([0,1], \Bbb{R}^m)$;
\par - a pair of parameters $(\alpha, \theta) \in \Bbb{R}^n \times L^p([0,1], \Bbb{R}^k).$
\\Put 
\begin{align*}
& X=W^{1,p} ([0,1],\Bbb{R}^n), \ U=L^p([0,1],\Bbb{R}^m), \ Z= X \times U,\\
&\Theta= L^p([0,1],\Bbb{R}^k), \ W=\Bbb{R}^n \times \Theta.
\end{align*}
Consider the constrained optimal control problem which depends on a parameters pair $(\alpha, \theta)$: \textit{Find a pair $(x,u)$, where $x\in W^{1,p}([0,1],\Bbb{R}^n)$ is a trajectory and $u \in L^p([0,1],\Bbb{R}^m)$ is a control function, which minimizes the objective function
\begin{align}\label{objective_Function_Control}
g(x(1)) + \int_0^1 L(t, x(t), u(t),\theta(t))dt
\end{align}
and satisfies the linear ordinary differential equation
\begin{align}
\label{state_equation}
\dot{x}(t)=A(t)x(t)+B(t)u(t)+C(t)\theta(t) \ \, \mbox{a.e.} \ t \in [0,1],
\end{align}
the initial value
\begin{align}
\label{initial_value}
x(0)=\alpha,
\end{align}
and the control constraint
\begin{align}
\label{control_constraint}
u \in \mathcal{U}.
\end{align}}

It is well-known that $X,\, U,\, Z,$ and $\Theta$ are Banach spaces. For each $w=(\alpha, \theta)\in W$, denote by $V(w)$ and $S(w)$, respectively, the optimal value and the solution set of problem~\eqref{objective_Function_Control}--\eqref{control_constraint}. We call $V: W \to \overline{\Bbb{R}}$, where $\overline{\Bbb{R}}=[-\infty,+\infty]$ denotes the extended real line, the \textit{optimal value function} of the problem in question. If for each $w=(\alpha, \theta)\in W$ we put
\begin{align}
\label{opjective_function_J}
J(x,u,w)=g(x(1))+ \int_0^1 L(t, x(t), u(t),\theta(t))dt,
\end{align}
\begin{align}
\label{map_constraint}
G(w)=\big\{z=(x,u)\in X \times U \,\mid \mbox{\eqref{state_equation} and \eqref{initial_value} are satisfied} \big \},
\end{align}
and
\begin{align}
\label{set_constraint}
K=X \times \mathcal{U},
\end{align}
then \eqref{objective_Function_Control}--\eqref{control_constraint} can be written formally as $\min\{J(z, w)\mid z \in G(w)\cap K\}$, and 
\begin{align}
\label{Re_optimal_value_function}
V(w)=\inf\{J(z, w)\mid z=(x,u)\in G(w)\cap K\}.
\end{align}

It is assumed that $V$ is finite at $\bar w=(\bar\alpha,\bar\theta)\in W$ and $(\bar x, \bar u)$ is a solution of the corresponding problem, that is $(\bar x, \bar u)\in S(\bar w)$. 

\medskip
Similarly as in \cite[p.~310]{KF70}, we say that a vector-valued function $g:[0,1]\to Y$, where $Y$ is a normed space, is essentially bounded if there exists a constant $\gamma >0$ such that the set $T:=\left\{t \in [0,1] \mid ||g(t)|| \le \gamma \right\}$ is of full Lebesgue measure. The latter means $\mu([0,1] \backslash\ T )=0,$ with $\mu$ denoting the Lebesgue measure on $[0,1].$

\medskip
Consider the following assumptions:
\begin{description}
\item[(A1)] The matrix-valued functions $A:[0,1] \to M_{n,n}(\Bbb{R}),$  $B:[0,1] \to M_{n,m}(\Bbb{R}),$ and $C:[0,1] \to M_{n,k}(\Bbb{R}),$ are measurable and essentially bounded.

\item[(A2)] The functions $g: \Bbb{R}^n \to \Bbb{R}$ and $L:[0,1]\times \Bbb{R}^n \times \Bbb{R}^m \times \Bbb{R}^k \to \Bbb{R}$ are such that  $g(\cdot)$ is convex and continuously differentiable on $\Bbb{R}^n$, $L(\cdot,x,u,v)$ is measurable for all $(x,u,v) \in \Bbb{R}^n \times \Bbb{R}^m \times \Bbb{R}^k$, $L(t,\cdot,\cdot,\cdot)$ is convex and continuously differentiable on $\Bbb{R}^n \times \Bbb{R}^m \times \Bbb{R}^k $ for almost every $t \in [0,1]$, and there exist constants $c_1>0,\, c_2>0,\, r \ge 0,\, p \ge p_1 \ge 0,\, p-1 \ge p_2 \ge 0$, and a nonnegative function $w_1\in L^p([0,1], \Bbb{R})$, such that $|L(t,x,u,v)| \le c_1\big (w_1(t) +||x||^{p_1}+ ||u||^{p_1}+||v||^{p_1}\big ),$
$$ \max \big \{ |L_x(t,x,u,v)|,|L_u(t,x,u,v)|,|L_v(t,x,u,v)| \big \} \le c_2\big (||x||^{p_2}+||u||^{p_2}+||v||^{p_2} \big )+r$$
for all $(t,x,u,v) \in [0,1] \times \Bbb{R}^n \times \Bbb{R}^m \times \Bbb{R}^k$.
\end{description}

We now recall some results from functional analysis related to Banach spaces. The results can be found in \cite[pp.~ 20-22]{Ioffe_Tihomirov_1979}.
For every $p\in [1,\infty)$, the symbol $L^p([0,1],\Bbb{R}^n)$ denotes the Banach space of Lebesgue measurable functions $x$ from $[0,1]$ to $\Bbb{R}^n$ for which the integral $\displaystyle\int_0^1 \|x(t)\|^p dt$ is finite. The norm in  $L^p([0,1],\Bbb{R}^n)$ is given by
$$||x||_p 
=\bigg(\int_0^1 \|x(t)\|^p dt \bigg)^{\frac{1}{p}}.$$
The dual space of $L^p([0,1],\Bbb{R}^n)$ is $L^q([0,1],\Bbb{R}^n)$, where $\frac{1}{p}+\frac{1}{q}=1.$ This means that, for every continuous linear functional $\varphi$ on the space $L^p([0,1],\Bbb{R}^n)$, there exists a unique element $x^* \in L^q([0,1],\Bbb{R}^n)$ such that $\varphi(x)=\langle \varphi, x \rangle =\int_0^1 x^*(t)x(t) dt$
for all $x\in L^p([0,1],\Bbb{R}^n)$. One has $||\varphi||=||x^*||_q.$

\medskip
The Sobolev space $W^{1,p} ([0,1],\Bbb{R}^n)$ is equipped with the norm 
$$||x||^1_{1,p}=\|x(0)\|+||\dot x||_p,$$ or the norm $||x||^2_{1,p}=||x||_p+||\dot x||_p.$ The norms $||x||^1_{1,p}$ and $||x||^2_{1,p}$ are equivalent (see, e.g., \cite[p.~21]{Ioffe_Tihomirov_1979}). Every continuous linear functional $\varphi$ on $W^{1,p} ([0,1],\Bbb{R}^n)$ can be represented in the form
$$\langle \varphi, x \rangle = \langle a, x(0)\rangle + \int_0^1 \dot {x}(t)u(t) dt,$$
where the vector $a\in \Bbb{R}^n$ and the function $ u\in L^{q} ([0,1],\Bbb{R}^n)$ are uniquely defined. In other words, $\big(W^{1,p} ([0,1],\Bbb{R}^n) \big)^*= \Bbb{R}^n\times L^{q} ([0,1],\Bbb{R}^n),$ where $\ \frac{1}{p}+\frac{1}{q}=1;$ see, e.g., \cite[p.~21]{Ioffe_Tihomirov_1979}. In the case of $p=2$, $W^{1,2} ([0,1],\Bbb{R}^n)$ is a Hilbert space with the inner product being given by
$$\langle x, y \rangle = \langle x(0), y(0) \rangle + \int_0^1 \dot x (t) \dot y (t)dt,$$
for all $x,y \in W^{1,2} ([0,1],\Bbb{R}^n).$

\medskip
We shall need two dual constructions from convex analysis \cite{Ioffe_Tihomirov_1979}. Let $X$ and $Y$ be Hausdorff locally convex topological vector spaces with the duals denoted, respectively, by $X^*$ and $Y^*$. For a convex set $\Omega\subset X$ and $\bar x\in \Omega$, the set
 \begin{align}\label{normals_convex_analysis} N(\bar x; \Omega)=\{x^*\in X^* \mid \langle x^*, x-\bar x \rangle \leq 0, \ \, \forall x \in \Omega\}\end{align} 
 is called the \textit{normal cone} of $\Omega$ at $\bar {x}.$
 Given a function $f: X\rightarrow \overline{\mathbb{R}}$, one says that $f$ is \textit{proper} if ${\rm{dom}}\, f:=\{ x \in X \mid f(x) < +\infty\} $
   is nonempty, and if $f(x) > - \infty$ for all $x \in X$. If $ {\rm{epi}}\, f:=\{ (x, \alpha) \in X \times \mathbb{R} \mid \alpha \ge f(x)\}$ is convex, then $f$ is said to be a convex function. The {\it subdifferential} of a proper convex function $f: X\rightarrow \overline{\mathbb{R}}$ at a point $\bar x \in {\rm dom}\,f$ is defined by
   \begin{align}\label{subdifferential_convex_analysis}
   \partial f(\bar x)=\{x^* \in X^* \mid \langle x^*, x- \bar x \rangle \le f(x)-f(\bar x), \ \forall x \in X\}.
   \end{align}	
It is clear that
  $\partial f(\bar x)=\{x^* \in X^* \mid (x^*,-1) \in N(( \bar x, f(\bar x)); {\rm{\epi}}\, f ) \}. $
  
  \medskip
  In the spirit of \cite[Definition~1.77]{Mordukhovich_2006a}, we define
  the {\it singular subdifferential} of a convex function $f$ at a point $\bar x \in {\rm dom}\,f$ by
   \begin{align}\label{singular_subdifferential_convex}
   \partial^\infty f(\bar x)=\{x^* \in X^* \mid (x^*,0)\in N ( (\bar x, f(\bar x)); {\rm{epi}}\, f)\}.
   \end{align} 
   For any convex function $f$, one has $\partial^\infty f(x)=N(x;{\rm dom}\,f)$ (see, e.g., \cite{AnYen}). If $\bar x\notin {\rm dom}\,f$, then we put $\partial f(\bar x)=\emptyset$ and $\partial^\infty f(\bar x)=\emptyset$.

\medskip   
 In the remaining part of this section, we present without proofs one theorem from \cite{AnToan} which is crucial for the subsequent proofs, thus making our exposition self-contained and easy for understanding.
 
 \medskip 
 Suppose that $X$, $W$, and $Z$ are Banach spaces with the dual spaces $X^*$, $W^*$, and $Z^*$, respectively. Assume that $M: Z\to X$ and $T: W\to X$
are continuous linear operators. Let $M^*: X^*\to Z^*$ and $T^*: X^*\to W^*$ be the adjoint operators of $M$ and $T$, respectively. Let $f:Z\times W\to\overline{\Bbb{R}}$ be a convex function and $\Omega$ a convex subset of $Z$ with nonempty interior. For each $w\in W$, put
 $ H(w)=\big\{z\in Z\,|\, Mz=Tw\big\}$ and consider the optimization problem 
  	 	\begin{eqnarray} \label{12b}
  	 	\min\{f(z,w) \, \mid \,z\in H(w)\cap \Omega \}.
  	 	\end{eqnarray} 
It is of interest to compute the subdifferential and the singular subdifferential of the optimal value function
  	 	\begin{eqnarray} \label{12}
  	 	h(w):=\inf_{z\in H(w)\cap \Omega} f(z,w)
  	 	\end{eqnarray} of \eqref{12b} where $w$ is subject to change. Denote by $\widehat S(w)$ the solution set of that problem. For our convenience, 
  	 we define the linear operator $\Phi: Z\times W \rightarrow X$ by setting $\Phi(z,w)=Mz-Tw$ for all $(z,w)\in Z \times W$. Note that 
  	the subdifferential $\partial h(w)$ has been computed in \cite{Chieu_Kien_Toan,AnToan}, while the singular subdifferential $\partial^\infty h(w)$ has been evaluated in \cite{AnToan}. 
  	
  	\medskip
  	The following result of \cite{AnToan} will be used intensively in this paper.
 
 	\begin{theorem}{\rm(See \cite[Theorem 2]{AnToan})}\label{asprogramingproblem} 
   	 		Suppose that $\Phi$ has closed range and ${\rm ker}\, T^* \subset {\rm ker}\,M^*$. If the optimal value function $h$ in \eqref{12} is finite at $\bar w \in {\rm dom}\, \widehat{S}$ and $f$ is continuous at $(\bar z,\bar w) \in  (W \times \Omega) \cap {\rm gph}\,H ,$ then
   	 		\begin{eqnarray}\label{BT1}
   	 		\partial h(\bar w)= \bigcup_{(z^*, w^*)\in \partial f(\bar z, \bar w)}\;\bigcup_{v^*\in N(\bar z;
   	 			\Omega)}\big[ w^*+ T^*\big((M^*) ^{-1}(z^* +v^*)\big)\big]
   	 		\end{eqnarray}
   	 		and
   	 		 		\begin{eqnarray}\label{BT1'}
   	 		\partial^\infty h(\bar w)= \bigcup_{(z^*, w^*)\in \partial^\infty f(\bar z, \bar w)}\;\bigcup_{v^*\in N(\bar z;
   	 			\Omega)}\big[ w^*+ T^*\big((M^*) ^{-1}(z^* +v^*)\big)\big],
   	 		\end{eqnarray}
   	 		where $\big(M^*) ^{-1}(z^* +v^*)=\{x^*\in X^*\mid M^*x^*=z^* +v^*\}$. 		
   	 	\end{theorem}
   	 	
 When $f$ is Fr\'echet differentiable at $(\bar z, \bar w)$, it holds that $\partial f(\bar z, \bar w)=\{\nabla f(\bar z, \bar w) \}$. Hence~\eqref{BT1} has a simpler form. Namely, the following statement is valid.
 \begin{theorem}\label{Frechet differentiable}
 	
 Under the assumptions of Theorem \ref{asprogramingproblem}, suppose additionally that the function $f$ is Fr\'echet differentiable at $(\bar z,\bar w)$. Then
 \begin{eqnarray}\label{BT1a}
    	 		\partial h(\bar w)= \bigcup_{v^*\in N(\bar z;
    	 			\Omega)}\big[ \nabla_w f(\bar z, \bar w)+ T^*\big((M^*) ^{-1}(\nabla_z f(\bar z, \bar w) +v^*)\big)\big],
    	 		\end{eqnarray}
 where $\nabla_z f(\bar z, \bar w)$ and $\nabla_w f(\bar{z}, \bar w)$, respectively, stand for the Fr\'echet derivatives of $f(\cdot,\bar w)$ at $\bar z$ and of $f(\bar z,\cdot)$ at $\bar w$.   	 		
 \end{theorem}
 	  
\section{The main results} \label{section3}
\markboth{\centerline{\it The main results}}{\centerline{\it D.T.V.~An, J.-C.~Yao,
and N.D.~Yen}} \setcounter{equation}{0}
Keeping the notation of Section 2, we consider the linear mappings $\mathcal{A}: X \to X$, $\mathcal{B}: U \to X,$ $\mathcal{M}: X \times U \to X,$ and $\mathcal{T}: W \to X$ which are defined by setting
\begin{align}
\label{mappingA}
\mathcal{A}x:= x- \int_0^{(.)} A(\tau) x(\tau) d \tau,
\end{align}
\begin{align}
\label{mappingB}
\mathcal{B}u:= - \int_0^{(.)} B(\tau) u(\tau) d \tau,
\end{align}
\begin{align}
\label{mappingM}
\mathcal{M}(x,u):=\mathcal{A}x +\mathcal{B}u,
\end{align}
and
\begin{align}
\label{mappingT}
\mathcal{T}(\alpha, \theta):= \alpha+ \int_0^{(.)} C(\tau) \theta(\tau) d \tau,
\end{align}
where the writing $\displaystyle\int_0^{(.)} g (\tau)d \tau$ for a function $g\in L^p([0,1], \Bbb{R}^n )$ abbreviates the function $t \mapsto \displaystyle\int_0^{t} g (\tau)d \tau$, which belongs to $X=W^{1,p}([0,1], \Bbb{R}^n ). $

\medskip
Under the assumption $\mathbf{(A1)}$, we can write the linear ordinary differential equation~\eqref{state_equation} in the integral form
$$x=x(0)+\int_0^{(.)} A(\tau) x(\tau) d \tau+\int_0^{(.)} B(\tau) u(\tau) d \tau+\int_0^{(.)} C(\tau) \theta(\tau) d \tau. $$
Combining this with the initial value in \eqref{initial_value}, one gets
$$x=\alpha+\int_0^{(.)} A(\tau) x(\tau) d \tau+\int_0^{(.)} B(\tau) u(\tau) d \tau+\int_0^{(.)} C(\tau) \theta(\tau) d \tau. $$
Thus, in accordance with \eqref{mappingA}--\eqref{mappingT},  \eqref{map_constraint} can be written as
\begin{equation}
\begin{split}
G(w)&=\left\{(x,u)\in X \times U \mid x= \alpha +\int_0^{(.)} A x d \tau+\int_0^{(.)} B u d \tau+\int_0^{(.)} C \theta d \tau \right\} \\
&=\left\{(x,u)\in X \times U \mid x- \int_0^{(.)} A x d \tau-\int_0^{(.)} B u d \tau= \alpha +\int_0^{(.)} C \theta d \tau \right\}\\
&=\left\{(x,u)\in X \times U \mid \mathcal{M}(x,u)=\mathcal{T}(w)\right\}.
\end{split}
\end{equation}
Hence, the control problem \eqref{objective_Function_Control}--\eqref{control_constraint}  reduces to the mathematical programming problem \eqref{12b}, where the function $J(\cdot)$, the multifunction $G(\cdot)$, and the set $K$ defined by \eqref{opjective_function_J}--\eqref{set_constraint}, play the roles of $f(\cdot)$, $H(\cdot)$, and $\Omega$.

\medspace

We shall need several lemmas.
\begin{lemma}{\rm{(See \cite[Lemma 2.3]{Toan_Kien}})}\label{Au_lemma1} Under the assumption $\mathbf{(A1)}$, the following are valid:
\\{\rm(i)} The linear operators $\mathcal{M}$ in \eqref{mappingM} and $\mathcal{T}$ in \eqref{mappingT} are continuous;
\\{\rm(ii)} $\mathcal{T}^*(a,u)=(a, C^Tu)$ for every $(a,u) \in \Bbb{R}^n \times L^q ([0,1], \Bbb{R}^n);$
\\{\rm(iii)} $\mathcal{M}^*(a,u)= (\mathcal{A}^*(a,u), \mathcal{B}^*(a,u ))$, where 
\begin{align}\label{formula_new}
\mathcal{A}^*(a,u)= \bigg( a- \int_0^{1} A^T(t) u(t) dt,\, u+ \int_0^{(.)} A^T(\tau) u(\tau) d \tau -   \int_0^{1} A^T(t) u(t) dt \bigg),
\end{align}
and $\mathcal{B}^*(a,u)=-B^Tu$ 
for every $(a,u)\in \Bbb{R}^n \times L^q ([0,1], \Bbb{R}^n).$
\end{lemma}

\begin{lemma} {\rm{(See \cite[Lemma 3.1 (b)]{Toan_Kien}})}
\label{Au_Lemma2} If $\mathbf{(A1)}$ and $\mathbf{(A2)}$ are satisfied, then the functional  $J: X \times U\times W \to \Bbb{R}$ is Fr\'echet differentiable at $(\bar z, \bar w)$ and $\nabla J(\bar z, \bar w)$ is given by
$$\nabla_w J(\bar z, \bar w)= \big (0_{\Bbb{R}^n}, L_\theta(\cdot, \bar x(\cdot), \bar u(\cdot), \bar \theta(\cdot) )\big ),$$
$$\nabla_z J(\bar z, \bar w)= \big(J_x(\bar x, \bar u,\bar \theta), J_u( \bar x, \bar u, \bar \theta )\big),$$
with 
\begin{align*} J_x( \bar x, \bar u, \bar \theta )&=\bigg( g'(\bar x(1)) + \int_0^1 L_x(t,\bar x(t),\bar u(t), \bar \theta(t))dt ,\\
& \quad \quad g'(\bar x(1)) + \int_{(.)}^1 L_x(\tau,\bar x(\tau),\bar u(\tau), \bar \theta(\tau))d\tau  \bigg),
 \end{align*}
 and $J_u( \bar x, \bar u, \bar \theta )=L_u(\cdot, \bar x(\cdot), \bar u(\cdot), \bar \theta(\cdot) ).$
 \end{lemma}

Let 
\begin{align*}
& \varPsi_A: L^q([0,1], \Bbb{R}^n ) \to \Bbb{R}, \quad \varPsi_B: L^q([0,1], \Bbb{R}^n ) \to L^q([0,1], \Bbb{R}^m ), \\
&\varPsi_C: L^q([0,1], \Bbb{R}^n ) \to L^q([0,1], \Bbb{R}^k ),\quad \varPsi :L^q([0,1], \Bbb{R}^n) \to L^q([0,1], \Bbb{R}^n )
\end{align*}
 be defined by 
 \begin{align*}
 \varPsi_A(v)= \int_0^1 A^T(t) v(t)dt,\quad \varPsi_B(v)(t)=-B^T(t)v(t) \ \, \mbox{a.e.}\ t \in [0,1],\\
 \varPsi_C(v)(t)=C^T(t)v(t)\ \, \mbox{a.e.} \  t \in [0,1], \quad \varPsi(v)=-\int_0^{(.)} A^T(\tau) v(\tau) d\tau.
  \end{align*}
  
    We will employ the following two assumptions.
  
  \begin{description}
 \item[(A3)]
Suppose that 
\begin{align}\label{re_con}
{\rm \ker}\, \varPsi_C \subset \big({\rm \ker}\, \varPsi_A \cap {\rm \ker}\, \varPsi_B \cap {\rm Fix}\, \varPsi\big),
\end{align}
 where ${\rm Fix}\, \varPsi:=\{x \in X \mid \varPsi(x)=x\}$ is the set of the fixed points of $\varPsi$, and ${\rm \ker}\, \varPsi_A$ (resp., ${\rm \ker}\, \varPsi_B$, ${\rm \ker}\, \varPsi_C$) denotes the kernel of $\varPsi_A$ (resp., $\varPsi_B$, $\varPsi_C$).
\item[(A4)] The operator $\varPhi: W \times Z \to X$, which is given by
$$\varPhi(w,z)=x - \int_0^{(.)} A(\tau) x(\tau) d \tau - \int_0^{(.)} B(\tau) v(\tau) d \tau -\alpha - \int_0^{(.)} C(\tau) \theta(\tau) d \tau $$
for every $ w=(\alpha, \theta) \in W$ and $z=(x,v) \in Z,$
has closed range. 
\end{description}
\begin{lemma}
\label{Au_lemma3}
If $\mathbf{(A3)}$ is satisfied, then ${\rm \ker}\, \mathcal{T}^* \subset {\rm \ker}\, \mathcal{M}^*.$
\end{lemma}
\begin{proof}
For any $(a,v) \in {\rm \ker}\, \mathcal{T}^*$, it holds that $(a,v) \in L^q ([0,1], \Bbb{R}^n)$ and $\mathcal{T}^*(a,v)=0$. Then $(a, C^Tv)=0$. So $a=0$ and $C^Tv=0$. The latter means that $C^T(t)v(t)=0$ a.e. on $[0,1]$. Hence $v \in {\rm \ker}\, \varPsi_C.$ By \eqref{re_con}, $$ v \in {\rm \ker}\, \varPsi_A \cap {\rm \ker}\, \varPsi_B \cap {\rm Fix}\, \varPsi. $$
The condition $ \varPsi_B(v)=0$ implies that $-B^T(t)v(t)=0$ a.e. on $[0,1]$. As $\mathcal{B}^*(a,v)=-B^Tv$ by Lemma \ref{Au_lemma1}, this yields
\begin{align}
\label{re_con1}
\mathcal{B}^*(a,v)=0.
\end{align}
According to the condition $v \in {\rm \ker}\, \varPsi_A$, we have $\displaystyle\int_0^1 A^T(t)v(t)dt=0$. Lastly, the condition $v \in {\rm Fix}\, \varPsi $ implies 
$v=-\displaystyle\int_0^{(.)} A^T(\tau) v(\tau) d \tau,$
hence $v+\displaystyle\int_0^{(.)} A^T(\tau) v(\tau) d \tau=0$. Consequently, using formula \eqref{formula_new}, we can assert that 
\begin{align}
\label{re_con2}
\mathcal{A}^*(a,v)=0.
\end{align}
Since $\mathcal{M}^*(a,v)=(\mathcal{A}^*(a,v), \mathcal{B}^*(a,v) )$ by Lemma \ref{Au_lemma1}, from \eqref{re_con1} and \eqref{re_con2} it follows that $\mathcal{M}^*(a,v)=0$. Thus, we have shown that ${\rm \ker}\, \mathcal{T}^* \subset {\rm \ker}\, \mathcal{M}^*.$
\end{proof}

The assumption $(H_3)$ in \cite{Toan_Thuy} can be stated as follows
\begin{description}
\item[(A5)] There exists a constant $c_3>0$ such that, for every $ v \in \Bbb{R}^n$,
$$||C^T(t)v|| \ge c_3 ||v||\ \, \mbox{a.e.} \ t \in [0,1]. $$
\end{description}
 \begin{Proposition}\label{Proposition_condition}
 If $\mathbf{(A5)}$ is satisfied, then $\mathbf{(A3)}$ and $\mathbf{(A4)}$ are fulfilled.
 \end{Proposition}
\begin{proof}
By $\mathbf{(A5)}$ and the definition of $\varPsi_C$, for any $v \in L^q([0,1], \Bbb{R}^n)$, one has
\begin{align*}
|| \varPsi_C (v)(t)|| \ge c_3 ||v(t)||\ \;\mbox{a.e.} \ t \in [0,1].
\end{align*}
So, if $\varPsi_C(v)=0$, i.e., $\varPsi_C(v)(t)=0$ a.e. $t \in [0,1]$, then $v(t)=0$ a.e. $t \in [0,1]$.
This means that ${\rm \ker}\, \varPsi_C=\{0\}$. Therefore, the condition \eqref{re_con} in $\mathbf{(A3)}$ is satisfied.

By Lemma \ref{Au_lemma1}, we have $\mathcal{T}^*(a,v)=(a, C^Tv)$ for every $(a,v) \in \Bbb{R}^n \times L^q ([0,1], \Bbb{R}^n)$. It follows that
\begin{align}
\label{formula2n}
||\mathcal{T}^*(a,v)||=||a||+||C^Tv||_q.
\end{align}
Using $\mathbf{(A5)}$, we get
\begin{align*}
||\mathcal{T}^*(a,v)||&=||a||+||C^Tv||_q \ge c_1 (||a||+||v||_q ),
\end{align*}
where $c_1=\min\{1, c_3\}.$ This means $||\mathcal{T}^*w^*|| \ge c_1 ||w^*||$ for every $w^* \in W^*$. By \cite[Theorem 4.13, p. 100]{Rudin_1991}, $\mathcal{T}:W \to X$ is surjective. Since $\varPhi(w,z)$ can be rewritten as $\varPhi(w,z)=-\mathcal{T}w+ \mathcal{M}z,$ the surjectivity of $\mathcal{T}$ implies that $\{\varPhi (w,0)\mid w\in W \}=X$. Hence $\varPhi$ has closed range.
\end{proof}

We are now in a position to formulate our main results on differential stability of problem \eqref{objective_Function_Control}--\eqref{control_constraint}. The following theorems not only completely describe the subdifferential and the singular subdifferential of the optimal value function, but also explain in detail the process of finding vectors belonging to the subdifferentials. In particular, from the results it follows that each subdifferential is either a singleton or an empty set.

\begin{theorem}
\label{Control_main_theorem}
Suppose that the optimal value function $V$ in \eqref{Re_optimal_value_function} is finite at $\bar w=(\bar \alpha, \bar \theta)$, ${\rm{int}\,} \mathcal{U} \not= \emptyset$, and $\mathbf{(A1)}-\mathbf{(A4)}$ are fulfilled. In addition, suppose that problem \eqref{objective_Function_Control}--\eqref{control_constraint}, with $\bar w=(\bar \alpha, \bar \theta)$ playing the role of $w=(\alpha, \theta)$, has a solution $(\bar x, \bar u).$ Then, a vector $(\alpha^*, \theta^*) \in \Bbb{R}^n \times L^q ([0,1], \Bbb{R}^k)$ belongs to $\partial \, V(\bar\alpha, \bar \theta)$ if and only if
\begin{align}
\label{eq1}
\alpha^*= g'(\bar x(1)) + \int_0^1 L_x(t,\bar x(t),\bar u(t), \bar \theta(t))dt-\int_0^1 A^T(t)y(t)dt,
\end{align}
\begin{align}
\label{eq5}
\theta^* (t) =-C^T (t)y(t)+L_\theta(t,\bar x(t),\bar u(t), \bar \theta(t))\ \, \mbox{a.e.}\ t\in [0,1],
\end{align}
where $y \in W^{1,q} ([0,1],\Bbb{R}^n)$ is the unique solution of the system 
\begin{align}\label{eq3}
\begin{cases}
\dot y (t) +A^T (t)y(t)=L_x(t,\bar x(t),\bar u(t), \bar \theta(t)) \ \rm\mbox{a.e.}\ t\in [0,1],\\
y(1)=-g'( \bar x(1)),
\end{cases}
\end{align}
such that the function $u^* \in L^q([0,1], \Bbb{R}^m)$ defined by
\begin{align}
\label{eq4}
 u^*(t)=B^T (t)y(t)-L_u(t,\bar x(t),\bar u(t), \bar \theta(t)) \ \mbox{a.e.}\ t\in [0,1]
\end{align}
satisfies the condition $u^* \in N(\bar u; \mathcal{U}).$ 
\end{theorem}
\begin{proof}
We apply Theorem \ref{Frechet differentiable} in the case where $J(z, w)$, $K$ and $ V(w)$, respectively, play the roles of $f(z,w)$, $\Omega$ and $h(w)$. By Lemmas \ref{Au_Lemma2} and \ref{Au_lemma3}, the conditions $\mathbf{(A1)}-\mathbf{(A4)}$ guarantee that all the assumptions of Theorem \ref{Frechet differentiable} are satisfied. So, we have
\begin{eqnarray}\label{BT1aa}
    	 		\partial V(\bar w)= \bigcup_{v^*\in N(\bar z;K)}\big[ \nabla_w J(\bar z, \bar w)+ \mathcal{T}^*\big((\mathcal{M}^*) ^{-1}(\nabla_z J(\bar z, \bar w) +v^*)\big)\big].
    	 		\end{eqnarray}
  From \eqref{BT1aa}, $(\alpha^*, \theta^*) \in \partial V(\bar w)$ if and only if
 \begin{align}
 \label{17new}
 (\alpha^*, \theta^*) - \nabla_w J(\bar z, \bar w) \in \mathcal{T}^*\big((\mathcal{M}^*) ^{-1}(\nabla_z J(\bar z, \bar w) +v^*)\big)
  \end{align}
 for some $v^* \in N(\bar z; K).$ Note that  $\nabla_w J(\bar z, \bar w)=(0_{\Bbb{R}^n},J_\theta(\bar z, \bar w) )$ and $v^*=(0_{\Bbb{R}^n},u^*)$ for some $u^* \in N(\bar u; \mathcal{U})$. Hence, from \eqref{17new} we get
  	 $$(\alpha^*, \theta^* -  J_\theta(\bar z, \bar w)) \in \mathcal{T}^*\big((\mathcal{M}^*) ^{-1}(\nabla_z J(\bar z, \bar w) +v^*)\big).$$	
Thus, there exists $(a,v) \in \Bbb{R}^n \times L^q ([0,1], \Bbb{R}^n)$	such that 
  	 \begin{align}
  	 \label{formula2}
  	 (\alpha^*, \theta^* -  J_\theta(\bar z, \bar w)) \in \mathcal{T}^*(a,v) \ \; \mbox{and}\ \;\nabla_z J(\bar z, \bar w) +v^*= \mathcal{M}^*(a,v).
  	 \end{align}
  	 By virtue of Lemma \ref{Au_lemma1}, we see that \eqref{formula2} is equivalent to the following
  	 \begin{align}
  	 & \begin{cases}
  	 \alpha^*=a, \ \theta^* -J_\theta (\bar z, \bar w)=C^T(\cdot )v(\cdot ),\\
  	 \big (J_x(\bar x, \bar u, \bar w), J_u(\bar x, \bar u,\bar w)+v^*  \big)=( \mathcal{A}^*(a,v), \mathcal{B}^*(a,v) ).
  	 \end{cases}\nonumber
  	  \end{align}
Invoking Lemma \ref{Au_Lemma2}, we can rewrite this system as	  
  	 \begin{align}
  	  \begin{cases}
  	 \alpha^*=a,\,	 \theta^*= L_\theta(\cdot, \bar x(\cdot ), \bar u(\cdot ), \bar \theta(\cdot )) + C^T(\cdot )v(\cdot ),\\
  	 g'( \bar x(1)) + \displaystyle\int_0^1L_x(t, \bar x(t), \bar u(t), \bar \theta (t)) dt =a - \displaystyle\int_0^1 A^T (t) v(t)dt,\\
  	 g'(\bar x(1) ) + \displaystyle\int _{(.)}^1 L_x (\tau, \bar x(\tau), \bar u (\tau ), \bar \theta(\tau))d \tau 
  	 = v(\cdot )+ \displaystyle\int_0^{(.)} A^T (\tau) v(\tau) d \tau-\displaystyle\int_0^1 A^T(t) v(t) dt,\\
  	 L_u(\cdot ,\bar x(\cdot ), \bar u(\cdot ), \bar \theta(\cdot)) +u^* =-B^T(\cdot )v(\cdot),
  	 \end{cases}\nonumber
  	 	  	   	\end{align}
 Clearly, the latter is equivalent to 	 	  	   	
  	  	\begin{align} \label{formula3}
  	  	 \begin{cases}
  	   	 \alpha^*=a,\,
  	   	 \theta^*= L_\theta(\cdot , \bar x(\cdot ), \bar u(\cdot ), \bar \theta(\cdot )) + C^T(\cdot )v(\cdot ),\\
  	   	 g'( \bar x(1)) + \displaystyle\int_0^1L_x(t, \bar x(t), \bar u(t), \bar \theta (t)) dt =a - \displaystyle\int_0^1 A^T (t) v(t)dt,\\
  	   	 g'(\bar x(1) ) - \displaystyle\int ^{(.)}_1 L_x (\tau, \bar x(\tau), \bar u (\tau ), \bar \theta(\tau))d \tau
  	   	  = v(\cdot)+ \displaystyle\int_1^{(.)} A^T (\tau) v(\tau) d \tau,\\
  	   	 L_u(\cdot ,\bar x(\cdot ), \bar u(\cdot ), \bar \theta(\cdot )) +u^* =-B^T(\cdot )v(\cdot).
  	   	 \end{cases}
  	 \end{align}
The third equality in \eqref{formula3} and the condition $v(\cdot)\in L^q([0,1], \Bbb{R}^n)$ imply that $v(\cdot)$ is absolutely differentiable on $[0,1]$ and, moreover, $\dot v(\cdot)\in L^q([0,1], \Bbb{R}^n)$. Hence $v(\cdot)\in W^{1,q}([0,1], \Bbb{R}^n)$. In addition, the third equality in \eqref{formula3} implies that $v(1)=g'(\bar x(1) )$. Moreover, by differentiating, we get $-\dot v(\cdot) - A^T (\cdot)v(\cdot) =L_x(\cdot, \bar x(\cdot), \bar u(\cdot), \bar \theta(\cdot)).$ Therefore,~\eqref{formula3} can be written as the following
\begin{align}\label{formula4}
  	\begin{cases}
	   	   	 \alpha^*=a,\,
	   	   v\in W^{1,q}([0,1], \Bbb{R}^n),\\	 
	      	\alpha^*= g'( \bar x(1)) + \displaystyle\int_0^1L_x(t, \bar x(t), \bar u(t), \bar \theta (t)) dt + \displaystyle\int_0^1 A^T (t) v(t)dt,\\
	      	\theta^*= L_\theta(\cdot, \bar x(\cdot), \bar u(\cdot), \bar \theta(\cdot)) + C^T(\cdot)v(\cdot) ,\\
	   	   	v(1)=g'(\bar x(1)),\\
	   	   	 -\dot v(\cdot) - A^T (\cdot)v(\cdot) =L_x(\cdot, \bar x(\cdot), \bar u(\cdot), \bar \theta(\cdot)),\\
	   	   	-B^T(\cdot)v(\cdot)= L_u(\cdot,\bar x(\cdot), \bar u(\cdot), \bar \theta(\cdot)) +u^* .
	   	   	 \end{cases}
	   	 \end{align}
Defining $y:=-v$ and omitting the vector $\alpha=\theta^* \in \Bbb{R}^n$, we can put \eqref{formula4} in the form 
	\begin{align*}
	  	  \begin{cases}
		   	   y\in W^{1,q}([0,1], \Bbb{R}^n),\\
		   	   \alpha^*= g'( \bar x(1)) + \displaystyle\int_0^1L_x(t, \bar x(t), \bar u(t), \bar \theta (t)) dt - \displaystyle\int_0^1 A^T (t) y(t)dt,\\
		   	   \theta^*(t)= L_\theta(t, \bar x(t), \bar u(t), \bar \theta(t)) + C^T(t)y(t) \ \,\mbox{a.e.}\ t\in [0,1],\\	 
		  	   	 \dot y(t) + A^T (t)y(t) =L_x(t, \bar x(t), \bar u(t), \bar \theta(t))\ \, \mbox{a.e.}\ t\in [0,1],\\
		  	   	  y(1)=-g'(\bar x(1)),\\
		   	   	B^T(t)y(t)-u^*(t)= L_u(t,\bar x(t), \bar u(t), \bar \theta(t)) \ \, \mbox{a.e.}\ t\in [0,1]. \\
		   	   	 \end{cases}
		   	 \end{align*}  
		   	 The assertion of the theorem follows easily from this system. 
\end{proof}
Next, let us show that how the singular subdifferential of $V(\cdot)$ can be computed.
\begin{theorem}
\label{control_singular_subdifferential} Suppose that all the assumptions of Theorem \ref{Control_main_theorem} are satisfied. Then, a vector $(\alpha^*, \theta^*) \in \Bbb{R}^n \times L^q ([0,1], \Bbb{R}^k)$ belongs to $\partial^\infty V(\bar w)$ if and only if
\begin{align}
\label{equa1}
\alpha^*=\int_0^1 A^T(t) v(t)dt,
\end{align}
\begin{align}
\label{equa2}
\theta^*(t)=C^T(t)v(t)\ \, \mbox{a.e.}\ t\in[0,1], 
\end{align}
where $v \in W^{1,q}([0,1], \Bbb{R}^n )$ is the unique solution of the system
\begin{align}
\label{equa3}
\begin{cases}
\dot v(t)=-A^T(t)v(t) \ \mbox{a.e.} \ t\in[0,1], \\
v(0)=\alpha^*,
\end{cases}
\end{align}
such that the function $u^* \in L^q([0,1], \Bbb{R}^n)$ given by 
\begin{align}
\label{equa4}
u^*(t)=-B^T(t)v(t) \ \mbox{a.e.} \ t\in[0,1]
\end{align}
belongs to $ N(\bar u, \mathcal{U}).$
\end{theorem}
\begin{proof}
We apply Theorem \ref{Frechet differentiable} in the case where $J(z, w)$, $K$ and $ V(w)$, respectively, play the roles of $f(z,w)$, $\Omega$ and $h(w)$. By Lemmas \ref{Au_Lemma2} and \ref{Au_lemma3}, the conditions $\mathbf{(A1)}-\mathbf{(A4)}$ guarantee that all the assumptions of Theorem \ref{asprogramingproblem} are satisfied. Hence, by \eqref{BT1'} we have
	\begin{eqnarray}\label{BT1'a}
 	\partial^\infty V(\bar w)= \bigcup_{(w^*, z^*)\in \partial^\infty J(\bar z, \bar w)}\;\bigcup_{v^*\in N(\bar z;
   	 			K)}\big[ w^*+ \mathcal{T}^*\big((\mathcal{M}^*) ^{-1}(z^* +v^*)\big)\big].
   	 		\end{eqnarray}
Since ${\rm \dom\,}J= Z \times W$ and $\partial^\infty J(\bar z, \bar w)=N((\bar z, \bar w);{\rm \dom\,}J )$ by \cite[Proposition 4.2]{AnYen}, it holds that $\partial^\infty J(\bar z, \bar w)=\{(0_{Z^*}, 0_{W^*}) \}$. Therefore, from \eqref{BT1'a} one gets
	\begin{eqnarray}\label{BT1'ab}
 	\partial^\infty V(\bar w)= \bigcup_{v^*\in N(\bar z;K)}\big[ \mathcal{T}^*\big((\mathcal{M}^*) ^{-1}(v^*)\big)\big].
   \end{eqnarray}
 Thus, a vector $(\alpha^*, \theta^*)\in \Bbb{R}^n \times L^q ([0,1], \Bbb{R}^k)$ belongs to $ \partial^\infty V(\bar w)$ if and only if one can find $v^* \in N(\bar z;K)$ and $(a,v)\in \Bbb{R}^n \times L^q([0,1], \Bbb{R}^n)$ such that
 \begin{align}
 \label{equa5} \mathcal{M}^*(a,v)=v^* \ \;\mbox{and} \ \;\mathcal{T}^* (a,v)=(\alpha^*, \theta^*).
 \end{align}
 Since $N(\bar z; K)=\{0_{\Bbb{R}^n}\} \times N(\bar u; \mathcal{U})$, we must have $v^*=(0_{\Bbb{R}^n}, u^*)$ for some $u^*\in N(\bar u; \mathcal{U})$. By Lemma \ref{Au_lemma1}, we can rewrite \eqref{equa5} equivalently as
  \begin{align*}
   \begin{cases}
 a-\displaystyle\int_0^1 A^T(t) v(t) dt=0,\\
 v(\cdot)+ \displaystyle\int_0^{(.)} A^T(\tau) v(\tau) d \tau - \displaystyle\int_0^1 A^T(t) v(t) dt=0,\\
 -B^T(\cdot)v(\cdot)=u^*(\cdot),\\
 a=\alpha^*,\\
 C^T(\cdot)v(\cdot)=\theta^*(\cdot).
 \end{cases}
 \end{align*}
 Omitting the vector $a \in \Bbb{R}^n$, we transform this system to the form
 \begin{align*}
  \begin{cases}
 \alpha^*=\displaystyle\int_0^1 A^T(t) v(t) dt,\\
 v(\cdot)=- \displaystyle\int_0^{(.)} A^T(\tau) v(\tau) d \tau + \displaystyle\int_0^1 A^T(t) v(t) dt,\\
 u^*(\cdot)=-B^T(\cdot)v(\cdot),\\
 \theta^*(\cdot)=C^T(\cdot)v(\cdot).
 \end{cases}  
 \end{align*}
 The second equality of the last system implies that $v \in W^{1,q} ([0,1], \Bbb{R}^n )$ (see the detailed explanation in the proof of Theorem \ref{Control_main_theorem}). Hence, that system is equivalent to the following
 \begin{align*}
  \begin{cases}
 v\in W^{1,q}([0,1], \Bbb{R}^n ),\\
  \alpha^*=\displaystyle\int_0^1 A^T(t) v(t) dt,\\
  \dot v(t)= -A^T(t) v(t) \ \,\mbox{a.e.} \ t \in [0,1],\\
  v(0)=\displaystyle\int_0^1 A^T(t) v(t) dt,
  \\
  u^*(t)=-B^T(t)v(t) \ \,\mbox{a.e.} \ t \in [0,1],\\
  \theta^*(t)=C^T(t)v(t) \ \,\mbox{a.e.} \ t \in [0,1].
  \end{cases}  
 \end{align*}
 These properties and the inclusion $u^*\in N(\bar u; \mathcal{U})$ show that the conclusion of the theorem is valid.
\end{proof}

\section{Illustrative examples}
\markboth{\centerline{\it Illustrative examples}}{\centerline{\it D.T.V.~An, J.-C.~Yao, and N.D.~Yen}} \setcounter{equation}{0}

We shall apply the results obtained in Theorems \ref{Control_main_theorem} and \ref{control_singular_subdifferential} to an  optimal control problem which has a clear mechanical interpretation.

\medskip
Following  Pontryagin et al. \cite[Example~1, p.~23]{Pontryagin_etal.__1962}, we consider a vehicle of mass~1 moving without friction on a straight road, marked by an origin, under the impact of a force $u(t)\in \mathbb{R}$ depending on time  $t\in [0,1]$. Denoting the coordinate of the vehicle at $t$ by $x_1(t)$ and its velocity by $x_2(t)$. According to Newton's Second Law, we have $u(t)=1\times \ddot{x}_1(t)$; hence
\begin{align}\label{P_EX}
\begin{cases}
\dot x_1(t)=x_2(t),\\
\dot x_2 (t)=u(t).
\end{cases}
\end{align} Suppose that the vehicle's initial coordinate and velocity are, respectively, $x_1(0)=\bar\alpha_1$ and $x_2(0)=\bar\alpha_2$. 
The problem is to minimize both the distance of the vehicle to the origin and its velocity at terminal time $t=1$. Formally, it is required that the sum of squares $[x_1(1)]^2+[x_2(1)]^2$ must be minimum when the measurable control $u(\cdot)$ satisfies the constraint $\displaystyle\int_0^1 |u(t)|^2dt\leq 1$ (\textit{an energy-type control constraint}). 

\medskip
It is worthy to stress that the above problem is different from the one considered in \cite[Example~1, p.~23]{Pontryagin_etal.__1962}, where \textit{the pointwise control constraint} $u(t)\in [-1,1]$ was considered and the authors' objective is to minimize the terminal time moment $T\in [0,\infty)$ at which $x_1(T)=0$ and $x_2(T)=0$. The latter conditions mean that the vehicle arrives at the origin with the velocity 0. As far as we know, the classical Maximum Principle \cite[Theorem~1, p.~19]{Pontryagin_etal.__1962} cannot be applied to our problem.

\medskip
We will analyze the model \eqref{P_EX} with the control constraint  $\displaystyle\int_0^1 |u(t)|^2dt\leq 1$ by using the results of the preceding section. Let $X= W^{1,2}([0,1], \Bbb{R}^2 )$, $ U= L^2([0,1], \Bbb{R} )$, $\Theta=  L^2([0,1], \Bbb{R}^2 ).$ Choose $A(t)=A$, $B(t)=B$, $C(t)=C$ for all $t\in [0,1]$, where  $$
A=\left(
\begin{array}{ll}
0&	1\\
0&	0\\
\end{array}   \right),
\quad
B=\left(
\begin{array}{l}
0\\
1\\
\end{array}   \right),  	
\quad  C=\left(
\begin{array}{ll}
1& 0\\
0 & 1\\
\end{array}   \right).	
$$
Put $g(x)=\|x\|^2$ for $x\in\mathbb{R}^2$ and $L(t,x,u,\theta)=0$ for $(t,x,u,\theta)\in [0,1]\times\mathbb{R}^2\times\mathbb{R}\times\mathbb{R}^2$. 
Let $\mathcal{U}=\left\{u \in L^2([0,1], \Bbb{R} ) \mid ||u||_2 \le 1\right\}.$ With the above described data set, the optimal control problem \eqref{objective_Function_Control}--\eqref{control_constraint} becomes
 \begin{align}\label{E_problem1n}
 \begin{cases}
 J(x,u,w)=x_1^2(1)+ x_2^2(1)\to {\rm \inf}\\
 \dot x_1(t)=x_2(t)+\theta_1(t),\
 \dot x_2(t)=u(t)+\theta_2(t),\\
 x_1(0)=\alpha_1, \,
 x_2(0)=\alpha_2,\,
 u \in \mathcal{U}.
 \end{cases}
 \end{align}
 The perturbation $\theta_1(t)$ may represent a noise in the velocity, that is caused by a small wind. Similarly, the perturbation $\theta_2(t)$ may indicate a noise in the force, that is caused by the inefficiency and/or improperness of the reaction of the vehicle's engine in response to a human control decision. We define the function $\bar\theta\in\Theta$ by setting $\bar\theta(t)=(0,0)$ for all $t\in [0,1]$. The vector $\bar\alpha=(\bar\alpha_1,\bar\alpha_2)\in\mathbb{R}^2$ will be chosen in several ways. 
 
 \medskip
 In next examples, optimal solutions of \eqref{E_problem1n} is sought for $\theta=\bar\theta$ and $\alpha=\bar\alpha$, where $\bar\alpha$ is taken from certain subsets of $\mathbb{R}^2$. These optimal solutions are used in the subsequent two examples, where we compute the subdifferential and the singular subdifferential of the optimal value function $V(w)$, $w=(\alpha,\theta)\in\mathbb R^2\times \Theta$, of \eqref{E_problem1n} at $\bar w=(\bar\alpha,\bar\theta)$  by applying Theorems \ref{Control_main_theorem} and \ref{control_singular_subdifferential}.

 \begin{ex}\label{EX1} \rm Consider the parametric problem \eqref{E_problem1n} at the parameter $w=\bar w$:
 	\begin{align}\label{E_problem1n(1)}
 	\begin{cases}
 	J(x,u, \bar w)=x_1^2(1)+ x_2^2(1)\to {\rm \inf}\\
 	\dot x_1(t)=x_2(t),\ \,
 	\dot x_2(t)=u(t),\\ x_1(0)=\bar\alpha_1, \,
 	x_2(0)=\bar\alpha_2,\,
 	u \in \mathcal{U}.
 	\end{cases}
 	\end{align}
 	In the notation of Section \ref{section3}, we interpret \eqref{E_problem1n(1)} as the parametric optimization problem
 	 \begin{align}\label{E_problem1n(2)}
 	  	\begin{cases}
 	  	J(x,u, \bar w)=x_1^2(1)+ x_2^2(1)\to {\rm \inf}\\
 	  	(x,u)\in G(\bar  w) \cap K,
 	  	 	\end{cases}
 	  	\end{align}
 	 where $G(\bar w)=\left\{(x,u)\in X \times U \mid \mathcal{M}(x,u)=\mathcal{T}(\bar w)\right\}$ and $K=X \times \mathcal{U}.$ 
Then, in accordance with \cite[Proposition~2, p.~81]{Ioffe_Tihomirov_1979}, $(\bar x, \bar u)$ is a solution of \eqref{E_problem1n(1)} if and only if
\begin{align}
\label{1plus}
(0_{X^*},0_{U^*})\in \partial_{x,u} J(\bar x, \bar u, \bar w) +N((\bar x, \bar u); G(\bar w) \cap K ).
\end{align}

\textit{Step 1} \big (computing the cone $N((\bar x, \bar u); G(\bar w))$\big). We have
\begin{align}
\label{normal_cone_G}
 N((\bar x, \bar u);G(\bar w))={\rm rge}(\mathcal{M^*}):=\{\mathcal{M^*}x^*\mid x^* \in X^* \}.
\end{align}
 Indeed, since $G(\bar w)=\left\{(x,u)\in X \times U \mid \mathcal{M}(x,u)=\mathcal{T}(\bar w)\right\}$ is an affine manifold,  \begin{align}\label{ker_mathcal_M} N((\bar x, \bar u);G(\bar w))&=({\rm ker}\mathcal{M})^\perp\\ \nonumber
  & =\{(x^*,u^*)\in X^* \times U^* \mid \langle (x^*,u^*), (x,u)\rangle=0\ \forall (x,u)\in {\rm ker}\mathcal{M}\}.\end{align}
  For any $z(\cdot)=(z_1(\cdot),z_2(\cdot))\in X$, if we choose $x_2(t)=z_2(0)$ and $x_1(t)=z_1(t)+z_2(0)t$ for all $t\in [0,1]$, and $u(t)=\dot x_2(t)$ for a.e. $t\in [0,1]$, then $(x,u)\in X\times U$ and $\mathcal{M}(x,u)=z$. This shows that the continuous linear operator $\mathcal{M}:X \times U\to X$ is surjective. In particular,  $\mathcal{M}$ has closed range. 
Therefore, by \cite[Proposition~2.173]{Bonnans_Shapiro_2000}, from \eqref{ker_mathcal_M} we get $$ N((\bar x, \bar u);G(\bar w))=({\rm ker}\mathcal{M})^\perp={\rm rge}(\mathcal{M^*})=\{\mathcal{M^*}x^*\mid x^* \in X^* \};$$ so \eqref{normal_cone_G} is valid.

\textit{Step 2} \big(decomposing the cone $N((\bar x, \bar u); G(\bar w ) \cap K)$\big).
To prove that 
\begin{align}
\label{decompose_normal}
N((\bar x, \bar u); G(\bar w ) \cap K)
  =\{0_{X^*}\} \times N(\bar u; \mathcal{U}) + N((\bar x, \bar u);G(\bar w)),
\end{align}
we first notice that 
\begin{align}
\label{normal_times}
N((\bar x, \bar u); K)
  =\{0_{X^*}\} \times N(\bar u; \mathcal{U}).
 \end{align} 
Next, let us verify the normal qualification condition
 \begin{align}
 \label{Q-C}
  N((\bar x, \bar u); K) \cap [- N((\bar x, \bar u);G(\bar w))]=\{(0,0)\}
 \end{align}
 for the convex sets $K$ and ${\rm gph}\, G$. Take any $(x_1^*,u_1^*) \in  N((\bar x, \bar u); K) \cap [- N((\bar x, \bar u);G(\bar w))].$ On one hand, by \eqref{normal_times} we have $x_1^*=0$ and $u_1^*\in N(\bar u; \mathcal{U})$. On the other hand, by \eqref{normal_cone_G} and the third assertion of Lemma~\ref{Au_lemma1}, we can find an element $$x^*=(a,v) \in X^*=\mathbb{R}^2 \times L^2([0,1], \mathbb{R}^2)$$ such that  $x_1^*= -\mathcal{A}^* (a,v)$ and $u_1^*=-\mathcal{B}^*(a,v).$ Then
 \begin{align}\label{system_n}
 0= \mathcal{A}^* (a,v),\ \, u_1^*= -\mathcal{B}^*(a,v).
  \end{align}
 Write $a=(a_1,a_2)$, $v=(v_1,v_2)$ with $a_i\in\mathbb R$ and $v_i\in  L^2([0,1], \mathbb{R})$, $i=1,\, 2$. According to Lemma~\ref{Au_lemma1}, \eqref{system_n} is equivalent to the following system
  \begin{align}\label{system2}
  \begin{cases}
  a_1=0,\, a_2-\displaystyle\int_0^1 v_1(t)dt=0,\\
   v_1=0, \\
   v_2+ \displaystyle\int_0^{(.)} v_1(\tau)d\tau-\int_0^1 v_1(t)dt=0,\\
  u_1^*=v_2.
  \end{cases}
 \end{align}
From \eqref{system2} it follows that $(a_1, a_2)=(0,0),$ $(v_1,v_2)=(0,0)$ and $u_1^*=0$. Thus $(x_1^*, u_1^*)=(0,0).$ Hence,~\eqref{Q-C} is fulfilled.
 
 Furthermore, since $\mathcal{U}=\left\{u \in L^2([0,1], \Bbb{R} ) \mid ||u||_2 \le 1\right\}$, we have ${\rm int}\, \mathcal{U}\not=\emptyset$; so $K$ is a convex set with nonempty interior. Due to \eqref{Q-C},  one cannot find any $(x_0^*, u_0^*) \in  N((\bar x, \bar u); K)$ and $(x_1^*,u_1^*)\in N((\bar x, \bar u);G(\bar w))$, not all zero, with $(x_0^*,u_0^*)+(x_1^*, u_1^*)=0.$ Hence, by \cite[Proposition~3, p.~206]{Ioffe_Tihomirov_1979}, $G(\bar w) \cap {\rm int}\, K\not=\emptyset$. Moreover, according to \cite[Proposition~1, p.~205]{Ioffe_Tihomirov_1979}, we have 
 $$ N((\bar x, \bar u); G(\bar w ) \cap K)
   =N((\bar x, \bar u);K) + N((\bar x, \bar u);G(\bar w)).$$
  Hence, combining the last equation with \eqref{normal_times} yields \eqref{decompose_normal}.
     
\textit{ Step 3} \big(computing the partial subdifferentials of $J( \cdot , \cdot, \bar w)$ at $(\bar x, \bar u)$\big ).
We first note that $J(x,u,\bar w)$ is a convex function. 
Clearly, the assumptions $\mathbf{(A1)}$ and $\mathbf{(A2)}$ are satisfied. Hence, by Lemma \ref{Au_Lemma2}, the function $J(x, u, \bar w )=g( x (1))=x_1^2(1)+ x_2^2(1)$ is Fr\'echet differentiable at $(\bar x,\bar u)$,  $J_u(\bar x, \bar u, \bar w)=0_{U^*}$, and \begin{align}\label{derivativeJ_0}J_x(\bar x, \bar u, \bar w)=\big(g'( \bar x(1)), g'(\bar x(1)) \big) =\big( (2 \bar x_1(1), 2 \bar x_2(1)), (2 \bar x_1(1), 2  \bar x_2(1)) \big),\end{align} where the first symbol $(2 \bar x_1(1), 2  \bar x_2(1)) $ is a vector in $\mathbb R^2$, while the second symbol $(2 \bar x_1(1), 2  \bar x_2(1)) $ signifies the constant function $t\mapsto  (2 \bar x_1(1), 2  \bar x_2(1)) $ from $[0,1]$ to $\mathbb R^2$. 
 Therefore, one has
\begin{equation}\label{derivativeJ}
\partial J_{x,u}(\bar x, \bar u, \bar w)= \left\{\big (J_x(\bar x, \bar u, \bar w), 0_{U^*} \big )\right\} 
\end{equation} with $J_x(\bar x, \bar u, \bar w)$ being given by \eqref{derivativeJ_0}. 

\smallskip
\textit{ Step 4} (solving the optimality condition)
By \eqref{normal_cone_G}, \eqref{decompose_normal}, and \eqref{derivativeJ}, we can assert that  \eqref{1plus} is fulfilled if and only if there exist  $u^* \in N(\bar u; \mathcal{U})$ and $x^*=(a,v) \in \mathbb{R}^2 \times L^2([0,1], \mathbb{R}^2 )$ with $a=(a_1,a_2)\in \mathbb{R}^2$, $v=(v_1,v_2)\in L^2([0,1], \mathbb{R}^2 )$, such that
\begin{align}
\label{4plus}
\big( \big ((-2 \bar x_1(1), -2 \bar x_2(1)), (-2  \bar x_1(1),- 2  \bar x_2(1))\big ), -u^* \big) =\mathcal{M^*}(a,v).
\end{align}
According to Lemma \ref{Au_lemma3}, we have $\mathcal{M^*}(a,v)=(\mathcal{A^*}(a,v), \mathcal{B^*} (a,v) ),$ where
\begin{align*}
\mathcal{A}^*(a,v)= \bigg( a- \int_0^{1} A^T(t) v(t) dt,\,  v+ \int_0^{(.)} A^T(\tau) v(\tau) d \tau -   \int_0^{1} A^T(t) v(t) dt \bigg),
\end{align*}
and $\mathcal{B}^*(a,v)=-B^Tv.$ Combining this with \eqref{4plus} gives
\begin{align}\label{4plusn}
\begin{cases}
-2 \bar x_1(1)=a_1,\ -2 \bar x_2(1)=a_2 - \displaystyle \int_0^1 v_1(t)dt,
\\
 -2 \bar x_1(1)=v_1,\, -2 \bar x_2(1)=v_2+\displaystyle\int_0^{(.)} v_1(\tau )d \tau - \displaystyle\int_0^1 v_1(t)dt,\\
u^*=v_2.
\end{cases}
\end{align}
If we can choose $a=0$ and $v=0$ for \eqref{4plusn}, then $u^*=0$; so $u^*\in N(\bar u; \mathcal{U})$. Moreover,~\eqref{4plusn} reduces to
\begin{align}\label{initial_constraints}
 \bar x_1(1)=0,\  \,\bar x_2(1)=0.
\end{align}
Besides, we observe that $(\bar x, \bar u) \in G(\bar w)$ if and only if
\begin{align}\label{5plus}
 	\begin{cases}
  	\dot {\bar x}_1(t)=\bar x_2(t),\ \,
 	\dot {\bar x}_2(t)=\bar u(t),\\ \bar x_1(0)=\bar\alpha_1, \,
 \bar 	x_2(0)=\bar\alpha_2,\,
 	\bar u\in \mathcal{U}.
 	\end{cases}
 	\end{align}
 Combining \eqref{initial_constraints} with \eqref{5plus} yields
 \begin{align}\label{6plus}
  	\begin{cases}
   	\dot {\bar x}_1(1)=0,\ \,
  	\dot {\bar x}_1(0)=\bar \alpha_2,\\ \bar x_1(0)=\bar\alpha_1, \,
    \bar x_1(1)=0,\\
    \dot {\bar x}_1(t)=\bar x_2(t),\ \,
     	\dot {\bar x}_2(t)=\bar u(t),\\
  	\bar u\in \mathcal{U}.
  	\end{cases}
  	\end{align}
  We shall find $\bar x_1(t)$ in the form $\bar x_1(t)=at^3+bt^2+ct+d$. Substituting this $\bar x_1(t)$ into the first four equalities in \eqref{6plus}, we get
  $$\begin{cases}
  3a+2b+c=0,\  c=\bar \alpha_2,\\
  d=\bar \alpha_1,\ a+b+c+d=0.
  \end{cases}$$
Solving this system, we have
$ a= 2\bar \alpha_1 + \bar \alpha_2,$ $
  b=-3 \bar \alpha_1 -2 \bar \alpha_2,$ $
  c=\bar \alpha_2,$ $
  d=\bar \alpha_1.
 $
Then $ \bar x_1(t)=(2\bar \alpha_1 + \bar \alpha_2)t^3- (3 \bar \alpha_1 +2 \bar \alpha_2 ) t^2 +\bar\alpha_2 t +\bar\alpha_1.$ So, from the fifth and the sixth equalities in  \eqref{6plus} it follows that
   \begin{align*}
 \begin{cases}
 \bar x_2(t)=\dot {\bar{x}}_1(t)=3(2\bar \alpha_1 + \bar \alpha_2)t^2- 2(3\bar  \alpha_1 +2\bar  \alpha_2 ) t +\bar\alpha_2 ,\\
 \bar u(t)=\dot {\bar{x}}_2(t)=(12\bar \alpha_1 + 6\bar \alpha_2)t- (6 \bar \alpha_1 +4\bar \alpha_2 ) .
 \end{cases}
 \end{align*}
Now, condition $\bar u \in \mathcal{U}$ in \eqref{6plus} means that
\begin{align}\label{7plus}
1 \ge \int_0^1 |\bar u(t)|^2dt &= \int_0^1 \left[(12\bar \alpha_1 +6 \bar \alpha_2)t- (6 \bar \alpha_1 +4\bar \alpha_2 ) \right]^2 dt.
\end{align}
By simple computation, we see that \eqref{7plus} is equivalent to
\begin{align}
\label{8plus}
12 \bar \alpha_1^2 + 12 \bar \alpha_1 \bar \alpha_2 +4 \bar \alpha_2^2 -1 \le 0.
\end{align}
Clearly, the set $\Omega$ of all the points $\bar\alpha=(\bar \alpha_1, \bar \alpha_2)\in\mathbb{R}^2$ satisfying \eqref{8plus} is an ellipse. 
We have shown that for every $\bar\alpha=(\bar \alpha_1, \bar \alpha_2)$ from $\Omega$, problem \eqref{E_problem1n(1)} has an optimal solution $(\bar x, \bar u)$, where 
   \begin{align}\label{optimal_solution}
   \begin{cases}
   \bar x_1(t)=(2\bar \alpha_1 + \bar \alpha_2)t^3-(3 \bar \alpha_1 +2 \bar \alpha_2 ) t^2 +\bar\alpha_2 t +\bar\alpha_1,\\
   \bar x_2(t)=3(2\bar \alpha_1 + \bar \alpha_2)t^2- 2(3\bar  \alpha_1 +2\bar  \alpha_2 ) t +\bar\alpha_2 ,\\
   \bar u(t)=(12\bar \alpha_1 +6 \bar \alpha_2)t- (6 \bar \alpha_1 +4\bar \alpha_2 ) .
   \end{cases}
   \end{align}
 In this case, the optimal value is $J(\bar x,\bar u'\bar w)=0.$ 
 \end{ex}
  
In the forthcoming two examples, we will use Theorems \ref{Control_main_theorem} and \ref{control_singular_subdifferential} to compute the subdifferential and the singular subdifferential of the optimal value function $V(w)$ of \eqref{E_problem1n} at $\bar w=(\bar\alpha,\bar\theta)$, where $\bar\alpha$ satisfies condition \eqref{8plus}. Recall that the set of all the points $\bar\alpha=(\bar \alpha_1, \bar \alpha_2)\in\mathbb{R}^2$ satisfying \eqref{8plus} is an ellipse, which has been denoted by $\Omega$. 

\begin{ex} \label{EX3} \rm \textit{(Optimal trajectory is implemented by an internal optimal control)} 
For $\alpha=\bar \alpha:=\left(\frac{1}{5}, 0 \right)$, that belongs to ${\rm int}\,\Omega$, and $\theta=\bar\theta$ with $\bar\theta(t)=(0,0)$ for all $t\in[0,1]$, 
the control problem \eqref{E_problem1n} becomes
\begin{align}\label{E_problem1}
\begin{cases}
J(x,u)=||x(1)||^2 \to {\rm \inf}\\
\dot x_1(t)=x_2(t),\
\dot x_2(t)=u(t),\\
x_1(0)=\frac{1}{5}, \,
x_2(0)=0,\,
u \in \mathcal{U}.
\end{cases}
\end{align}
 For the parametric problem \eqref{E_problem1n}, it is clear that the assumptions $\mathbf{(A1)}$ and $\mathbf{(A2)}$ are satisfied. As $C(t)=\left(
  \begin{array}{ll}
  1&	0\\
  0&	1\\
  	\end{array}
   \right)$ for $t\in\ [0,1]$, one has for every $v \in \Bbb{R}^2$ the following 
   $$||C^T(t)v||=||v||\ \, {\rm \mbox{a.e.}}\ t\in [0,1].$$ Hence, assumption $\mathbf{(A5)}$ is also satisfied. Then, by Proposition \ref{Proposition_condition}, the assumptions $\mathbf{(A3)}$ and $\mathbf{(A4)}$ are fulfilled. According to \eqref{optimal_solution} and the analysis given in Example \ref{EX1}, the pair $(\bar x, \bar u)\in X\times U$, where $\bar x(t)= \left( \frac{2}{5} t^3-\frac{3}{5}t^2 +\frac{1}{5} , \frac{6}{5}t^2 -\frac{6}{5}t \right)$ and $\bar u(t)=\frac{12}{5}t-\frac{6}{5}$ for $t\in [0,1]$, is a solution of \eqref{E_problem1}. In this case, $\bar u(t)$ is an interior point of $\mathcal{U}$ since $\int_0^1|\bar u(t)|^2dt =\frac{12}{25}<1$.
 Thus, by Theorem \ref{Control_main_theorem}, a vector $(\alpha^*, \theta^*) \in \Bbb{R}^2 \times L^2 ([0,1], \Bbb{R}^2)$ belongs to $\partial \, V(\bar\alpha, \bar \theta)$ if and only if
   \begin{align}
   \label{eq11}
   \alpha^*= g'(\bar x(1)) -\int_0^1 A^T(t)y(t)dt
   \end{align}
   and
   \begin{align}
   \label{eq55}
   \theta^* (t) =-C^T (t)y(t)\ \,\mbox{a.e.}\ t\in [0,1],
   \end{align}
   where $y \in W^{1,2} ([0,1],\Bbb{R}^2)$ is the unique solution of the system 
   \begin{align}\label{eq33}
   \begin{cases}
   \dot y (t) =-A^T (t)y(t)\ \, \mbox{a.e.}\ t\in [0,1],\\
   y(1)=-g'( \bar x(1)),
   \end{cases}
   \end{align}
   such that the function $u^* \in L^2([0,1], \Bbb{R})$ defined by
   \begin{align}
   \label{eq44}
    u^*(t)=B^T (t)y(t)\ \,\mbox{a.e.}\ t\in [0,1]
   \end{align}
   satisfies the condition $u^* \in N(\bar u; \mathcal{U}).$ 
   
   Since $\bar x(1)= \left(0,0 \right)$, 
            we have $g'( \bar x(1))=(0,0)$.          
 So, \eqref{eq33} can be rewritten as 
    \begin{align*}
      \begin{cases}
      \dot y_1 (t) =0,\
      \dot y_2(t)=-y_1(t),\\
      y_1(1)=0,\
      y_2(1)=0.
      \end{cases}
      \end{align*}
Clearly, $y(t)=\left(0,0
            \right)$ is the unique solution of this terminal value problem . Combining this with \eqref{eq11}, \eqref{eq55} and \eqref{eq44}, we obtain $\alpha^*= \left(0,0 \right)$ and $\theta^*(t)=\theta^*=(0,0)$ a.e. $t\in [0,1],$ and $u^*(t)=0$ a.e. $t \in [0,1]$. Since
            $u^*(t)=0$ satisfies the condition $u^* \in N(\bar u; \mathcal{U})$, we have $\partial V(\bar w)=\{(\alpha^*, \theta^*)\}$, where $\alpha^*= \left(0,0 \right)$ and $\theta^*=(0,0)$.
  
  We now compute $\partial V^\infty(\bar \alpha, \bar \theta)$. By Theorem \ref{control_singular_subdifferential}, a vector $(\tilde\alpha^*, \tilde\theta^*) \in \Bbb{R}^2 \times L^2 ([0,1], \Bbb{R}^2)$ belongs to $\partial^\infty V(\bar w)$ if and only if
\begin{align}
\label{equa11}
\tilde\alpha^*=\int_0^1 A^T(t) v(t)dt,
\end{align}
\begin{align}
\label{equa22}
\tilde\theta^*(t)=C^T(t)v(t)\ \,\mbox{a.e.}\ t\in[0,1], 
\end{align}
where $v \in W^{1,2}([0,1], \Bbb{R}^2 )$ is the unique solution of the system
\begin{align}
\label{equa33}
\begin{cases}
\dot v(t)=-A^T(t)v(t)\ \,\mbox{a.e.} \ t\in[0,1],\\
v(0)=\tilde\alpha^*,
\end{cases}
\end{align}
such that the function $\tilde u^* \in L^2([0,1], \Bbb{R})$ given by 
\begin{align}
\label{equa44}
\tilde u^*(t)=-B^T(t)v(t)\ \;\mbox{a.e.} \ t\in[0,1]
\end{align}
belongs to $ N(\bar u; \mathcal{U}).$  Thanks to \eqref{equa11}, we can rewrite \eqref{equa33} as
 \begin{align*}
      \begin{cases}
      \dot v_1 (t) =0,\
      \dot v_2(t)=-v_1(t),\\
      v_1(0)=0,\
      v_2(0)=\displaystyle\int_0^1 v_1(t)dt.
      \end{cases}
      \end{align*}
  It is easy to show that $v(t)=(0,0)$ is the unique solution of this system. Hence, \eqref{equa11}, \eqref{equa22} and \eqref{equa44} imply that $\tilde\alpha^*=(0,0)$, $\tilde\theta^*=(0,0)$ and $\tilde u^*=0$. Since  $\tilde u^* \in N(\bar u; \mathcal{U})$, we have  $\partial^\infty V(\bar w)=\{(\tilde\alpha^*,\tilde\theta^*)\}$, where $\tilde\alpha^*= \left(0,0 \right)$ and $\tilde\theta^*=(0,0)$.
\end{ex}

\begin{ex} \label{EX4} \rm \textit{(Optimal trajectory is implemented by a boundary optimal control)}
For $\alpha=\bar \alpha:=\left(0, \frac{1}{2}\right)$, that belongs to ${\rm \partial}\,\Omega$, and $\theta=\bar\theta$ with $\bar\theta(t)=(0,0)$ for all $t\in[0,1]$, problem \eqref{E_problem1n} becomes
\begin{align}\label{E_problem4}
\begin{cases}
J(x,u)=||x(1)||^2 \to {\rm \inf}\\
\dot x_1(t)=x_2(t),\
\dot x_2(t)=u(t),\\
x_1(0)=0, \,
x_2(0)=\frac{1}{2},\,
u \in \mathcal{U}.
\end{cases}
\end{align}
As it has been shown in Example \ref{EX1},  $(\bar x, \bar u) =\left(\frac{1}{2} t^3-t^2 +\frac{1}{2}t , \frac{3}{2}t^2-2t + \frac{1}{2}, 3t-2\right)$ is a solution of \eqref{E_problem4}. In this case, we have $\int_0^1 |\bar u(t)|^2dt=\int_0^1 (3t-2)^2dt=1$. This means that $\bar u(t)$ is a boundary point of $\mathcal{U}.$ So, $N(\bar u; \mathcal{U})=\{\lambda \bar u\mid \lambda\geq 0\}$.  Since $\bar x(1)= \left(0,0 \right)$, arguing in the same manner as in Example \ref{EX3}, we obtain $\partial V(\bar w)=\left\{ \left(\alpha^*,\theta^*\right)\right\}$ and $\partial^\infty V(\bar w)=\left\{\left(\tilde\alpha^*,\tilde\theta^*\right)\right\}$, where  $\alpha^*=\tilde\alpha^*=\left(0, 0\right)$ and $\theta^*=\tilde\theta^*=\left(0, 0\right)$.
\end{ex}
\section*{Acknowledgements}

\noindent
This work was supported by College of Sciences, Thai Nguyen University (Vietnam), the Grant MOST 105-2221-E-039-009-MY3 (Taiwan), and  National Foundation for Science $\&$ Technology Development (Vietnam).

\end{document}